\documentclass[11pt]{amsart}
\usepackage{amssymb,amsmath,amsthm,amsfonts,amsopn,url,color,hyperref,enumitem}
\usepackage[all]{xy}

\theoremstyle{plain}
\newtheorem{thm}{Theorem}[section]
\newtheorem{prop}[thm]{Proposition}
\newtheorem{lemma}[thm]{Lemma}
\newtheorem{cor}[thm]{Corollary}

\theoremstyle{definition}
\newtheorem{defn}[thm]{Definition}
\newtheorem*{defn*}{Definition}
\newtheorem*{question*}{Question}

\newtheorem{example}[thm]{Example}
\newtheorem*{example*}{Example}
\newtheorem{rem}[thm]{Remark}
\newtheorem*{rem*}{Remark}
\newcommand{\field}[1]{\mathbb{#1}}
\newcommand{\N}{\field{N}}
\newcommand{\Z}{\field{Z}}

\newcommand{\R}{\field{R}}
\newcommand{\F}{\field{F}}

\newcommand{\ideal}[1]{\mathfrak{#1}}
\newcommand{\m}{\ideal{m}}
\newcommand{\n}{\ideal{n}}
\newcommand{\p}{\ideal{p}}
\newcommand{\q}{\ideal{q}}

\newcommand{\func}[1]{\mathrm{#1} \,}

\newcommand{\Spec}{\func{Spec}}

\newcommand{\hgt}{\func{ht}}

\newcommand{\ra}{\rightarrow}

\DeclareMathOperator{\ann}{ann}

\newcommand{\be}{\begin{enumerate}}
\newcommand{\ee}{\end{enumerate}}

\newcommand{\li}
 {\leftfootline}

\renewcommand{\phi}{\varphi}

\DeclareMathOperator{\Max}{Max}

\DeclareMathOperator{\orc}{c}
\newcommand{\frk}{\mathfrak{K}}
\DeclareMathOperator{\sK}{sK}

\author{Neil Epstein}
\address{Department of Mathematical Sciences \\ George Mason University \\ Fairfax, VA  22030}
\email{nepstei2@gmu.edu}

\author{Jay Shapiro}
\address{Department of Mathematical Sciences \\ George Mason University \\ Fairfax, VA  22030}
\email{jshapiro@gmu.edu}

\title[The Ohm-Rush content function II]{The Ohm-Rush content function II. Noetherian rings, valuation domains, and base change}

\date{May 9, 2018}
\begin{document}
\begin{abstract}
The notion of an Ohm-Rush algebra, and its associated content map, has connections with prime characteristic algebra, polynomial extensions, and the Ananyan-Hochster proof of Stillman's conjecture.   As further restrictions are placed (creating the increasingly more specialized notions of weak content, semicontent, content, and Gaussian algebras), the construction becomes more powerful. Here we settle the question in the affirmative over a Noetherian ring from \cite{nmeSh-OR} of whether a faithfully flat weak content algebra is semicontent (and over an Artinian ring of whether such an algebra is content), though both questions remain open in general.  We show that in content algebra maps over Pr\"ufer domains, heights are preserved and a dimension formula is satisfied.  We show that an inclusion of nontrivial valuation domains is a content algebra if and only if the induced map on value groups is an isomorphism, and that such a map induces a homeomorphism on prime spectra.  Examples are given throughout, including results that show the subtle role played by properties of transcendental field extensions.
\end{abstract}

\maketitle

\section{Introduction}\label{sec:intro}

In the 1970s, Ohm and Rush \cite{OhmRu-content, Ru-content} came up with an axiomatic theory to determine how close a faithfully flat ring map $R \rightarrow S$ is to acting like the polynomial extension map $R \rightarrow R[x]$.  The key idea is to generalize the notion of the ``content'' of a polynomial to an element of $S$ and then see which formulas the function satisfies.  In increasing order of specificity, one may ask (with some updated terminology) whether a faithfully flat $R$-algebra $S$ is \begin{enumerate}
\item Ohm-Rush (from \cite{OhmRu-content}, terminology from \cite{nmeSh-OR}),
\item weak content \cite{Ru-content},
\item semicontent \cite{nmeSh-OR},
\item content \cite{OhmRu-content}, or 
\item Gaussian \cite{Nas-ABconj}.
\end{enumerate}

Dedekind and Mertens \cite{Ded-DM, Mer-DM} essentially showed that the polynomial extension is a content algebra.  Gauss classically showed that the map $\Z \rightarrow \Z[x]$ is Gaussian, and Pr\"ufer \cite{Pr-DM} showed a polynomial extension is Gaussian whenever the base ring $R$ is a Pr\"ufer domain.  When $R$ is an integral domain, a converse was shown independently by Tsang \cite{Ts-Gauss} and Gilmer \cite{Gil-hilf}.
We showed that the power series extension map $R \rightarrow R[[x]]$ is a content algebra whenever $R$ is Noetherian \cite{nmeSh-DMpower}, but only in special cases when $R$ is a valuation ring \cite[Corollary 4.11]{nmeSh-OR}.

The question arises then of when some or all of these conditions are equivalent.  We have no examples among faithfully flat ring maps to demonstrate that any two of  the middle three properties are distinct.  However, we show that if $R$ is Noetherian and $R \rightarrow S$ is weak content, it is semicontent (see Corollary~\ref{weaksemi}).  Over an Artinian ring, conditions 2-4 are equivalent (see Theorem~\ref{thm:Artequiv}).  In the case where R is a Pr\"ufer domain, even conditions 2-5 are equivalent \cite[Proposition 4.7]{nmeSh-OR}.

As is often the case, more structural information is available in the case where the base ring is (or both rings are) Pr\"ufer or Dedekind.  We get a particularly nice description of the content function in the case of extensions of the form $K[x] \rightarrow L[x]$, where $L/K$ is a field extension.

In another strand of research, the Ohm-Rush property has also been called ``intersection-flatness'' \cite[p. 41]{HHbase} and has been used to investigate properties of the Frobenius endomorphism in prime characteristic algebra \cite{HHbase, Ka-param, Sha-nonred}.  Outside of prime characteristic algebra, some of these ideas have even been useful in the recent solution \cite{AnHo-small} to Stillman's conjecture \cite{PeSt-open} about bounds on projective dimension.  It is noted in \cite[p.41]{HHbase} that if $R$ is complete Noetherian local and $R \rightarrow S$ is a flat local homomorphism,  then it follows from Chevalley's theorem \cite[Lemma 7]{Che-local} that $S$ is an intersection-flat (i.e. Ohm-Rush) $R$-algebra.

An analogue of the Ohm-Rush content function has also recently been analyzed  \cite{Nas-consemi, Nas-zdsemi} in semirings (i.e. commutative ring-like structures where subtraction is impossible or disallowed).

The structure of this paper is as follows.  In \S\ref{sec:basics}, we recall known information about Ohm-Rush content.  In \S\ref{sec:localglobal} we establish that the properties of being a semicontent algebra, a content algebra, a weak content algebra, and in some cases also a content algebra, are local.  See Propositions~\ref{localsemicontent}--\ref{localcontent}.  We also establish there that base change by a factor of the base ring is harmless. See Lemma~\ref{lem:factOR} and Proposition~\ref{pr:factorrings}.  In \S\ref{sec:NoethArt}, after proving a general fact (see Theorem~\ref{thm:primary}) about primary ideals extending to primary ideals in the context of Noetherian base rings and INC extensions, we conclude in Corollary~\ref{weaksemi} that for faithfully flat maps over a Noetherian base ring, the weak content and semicontent properties coincide.  We then connect this with some results of Ananyan and Hochster in Corollary~\ref{cor:AH}.

The above connection leads naturally to an investigation of when base change of a polynomial ring over a field by an extension field is a weak content, semicontent, or content algebra.  This in turn depends on subtle properties of field extensions.  See Corollary~\ref{cor:aclosed}, Proposition~\ref{pr:regext}, Lemma~\ref{lem:regext}, Proposition~\ref{pr:purecontent}, and the limiting Example~\ref{ex:aclosed}. We then show in Proposition~\ref{pr:weaksemiINC} that the weak and semicontent properties also coincide for faithfully flat INC extensions.  The section then concludes in showing that the content, semicontent, and weak content properties coincide for faithfully flat extensions of an Artinian base ring  (see Theorem~\ref{thm:Artequiv}), even though the Gaussian property does not (see Example~\ref{ex:Art}).  

Section~\ref{sec:valbase} consists of an analysis of the Ohm-Rush content function in the case where the base ring is (locally) a valuation domain.  After first giving a criterion for Ohm-Rushness over a valuation domain (see Proposition~\ref{pr:valOR}), we show in Theorem~\ref{thm:height} that heights of primes are preserved in content algebra maps over Pr\"ufer domains.  In the special case where the target ring is catenary, this allows us to conclude a dimension formula (cf. Corollary~\ref{cor:dimcatenary}) that aligns with other dimension formulas in the literature.

In the final section \S\ref{sec:valval}, we further specialize to the case where \emph{both} the base and target rings are (locally) valuation domains.  A key technical result is that in a content algebra map of nontrivial Pr\"ufer domains, maximal ideals always extend to maximal ideals (see Corollary~\ref{cor:maxextend}).  Next we obtain a clean formula for the content of an element in a content algebra map of valuation rings (see Proposition~\ref{pr:cval}).  This leads us to Example~\ref{ex:polynomial} and Proposition~\ref{pr:irredfield}, where we see that $K[x] \rightarrow L[x]$ is a content algebra if and only if $K$ is algebraically closed in $L$, and we obtain a clean formula for the content function in this case.  In Theorem~\ref{thm:valgroups}, we show that a map of valuation rings is a content algebra only if it induces an isomorphism on value groups.  In Theorem~\ref{thm:homeo}, we show that any such map induces a homeomorphism on prime spectra.  We close with a method for constructing another class of content algebras that do not fall into any of the preceding categories (see Example~\ref{ex:end}).

\section{Basics}\label{sec:basics}
We commence with the basic definitions of this investigation.
\begin{defn}\label{def:orc}
Let $R$ be a ring, $M$ an $R$-module, and $f\in M$.  Then the \emph{(Ohm-Rush) content} of $f$ (described in \cite{OhmRu-content}; current nomenclature from \cite{nmeSh-OR}) is given by  \[
\orc(f) := \bigcap\{I \subseteq R \text{ ideal } \mid f \in IM\}.\footnote{In \cite{nmeSh-OR}, we use the symbol $\Omega$ for this function.}
\]
If $f\in \orc(f)M$ for all $f\in M$, we say that $M$ is an \emph{Ohm-Rush module}; if $M$ is moreover an $R$-algebra, we say that it is an \emph{Ohm-Rush algebra} over $R$.
\end{defn}

\begin{defn}\label{defs}
Let $R \rightarrow S$ be an Ohm-Rush algebra.  We say that it is  \begin{enumerate}
\item a \emph{weak content algebra} \cite{Ru-content} if $\sqrt{\orc(fg)} = \sqrt{\orc(f)\orc(g)}$ for all $f, g \in S$,
\item a \emph{semicontent algebra} \cite{nmeSh-OR} if it is faithfully flat and for any multiplicative set $W \subseteq R$, whenever $f,g \in S$ with $\orc(f)_W = R_W$, we have $\orc(fg)_W = \orc(g)_W$,
\item a \emph{content algebra} \cite{OhmRu-content} if it is faithfully flat and for any $f, g \in S$, there is some $n\in \N$ with $\orc(f)^n \orc(g) = \orc(f)^{n-1} \orc(fg)$,
\item a \emph{Gaussian algebra} \cite{Nas-ABconj} if it is faithfully flat and for any $f, g \in S$, we have $\orc(fg) = \orc(f) \orc(g)$.
\end{enumerate}
\end{defn}

We recall some properties in the following proposition.

\begin{prop}\label{pr:omnibus}
Let $R$ be a commutative ring, and let $L$ be an $R$-module.  \begin{enumerate}
\item\label{it:OR12ii} \cite[1.2(ii)]{OhmRu-content} $L$ is an Ohm-Rush $R$-module if and only if for all collections $\{I_\alpha\}$ of ideals of $R$, we have \[
\left(\bigcap_\alpha I_\alpha\right)L = \bigcap_\alpha I_\alpha L.
\]
\item\label{it:ORflat}\cite[Corollary 1.6]{OhmRu-content} Let $L$ be an Ohm-Rush $R$-module.  It is flat over $R$ if and only if for all $x\in L$ and $r\in R$, we have $\orc(rx) = r\orc(x)$.  A flat Ohm-Rush module $L$ is faithfully flat if and only if no proper ideal of $R$ contains all ideals of the form $\orc(x)$, $x\in L$.
\item\label{it:OR31}\cite[Theorem 3.1]{OhmRu-content} Let $L$ be a flat Ohm-Rush module, and $W$ a multiplicatively closed subset of $R$.  Then $L_W$ is an Ohm-Rush $R_W$-module, and for any $g\in L$ and $w\in W$, we have $\orc_W(g/w) = \orc(g)_W$, where $\orc$ is the content function on $L$ with respect to $R$ and $\orc_W$ is the content function on $L_W$ with respect to $R_W$.
\item\label{it:OR62}\cite[Theorem 6.2]{OhmRu-content} Let $R \rightarrow S$ be a content algebra.  Let $W$ be a multiplicatively closed subset of $S$ such that for all $w\in W$, we have $\orc(w)\cap W \neq \emptyset$.  Then $S_W$ is a content $R_{W \cap R}$-algebra, and for any $g\in S$ and $w\in W$, we have $\orc_W(g/w) = \orc(g)_{W \cap R}$.
\item\label{it:Ruprimes}\cite[Theorem 1.2]{Ru-content} Let $R \rightarrow S$ be an Ohm-Rush algebra.  It is a weak content algebra if and only if for all $\p \in \Spec R$, either $\p S=S$ or $\p S \in \Spec S$.
\item\label{it:Rumult}\cite[Proposition 1.1(i)]{Ru-content} Let $R \rightarrow S$ be an Ohm-Rush algebra and $f, g\in S$.  Then $\orc(fg) \subseteq \orc(f)\orc(g)$. 
\end{enumerate}
\end{prop}

We also recall that if $R \subseteq S$ is a faithfully flat ring map, then it is \emph{pure}, which implies that for any ideal $J$ of $R$, we have $JS \cap R = J$ (see \cite[Theorem 7.5]{Mats}).  The following result is \cite[Proposition 2.6]{nmeSh-OR}.

\begin{prop}
Any ring extension $R \subset S$ satisfies the following implications:
\begin{center}
(faithfully flat) Gaussian algebra $\implies$ content algebra \\
$\implies$ semicontent algebra $\implies$ weak content algebra.
\end{center}
\end{prop}

The above result implies that in a semicontent algebra $S$ over $R$, prime ideals of $R$ extend to prime ideals of $S$.  In fact, more can be said.
The proof of the following is given in \cite[Proposition 2.7]{nmeSh-OR}, though the proposition itself is misstated there.

\begin{prop}\label{pr:primaries}
Let $R$ be a ring, and let $S$ be a faithfully flat Ohm-Rush algebra.  If $ S$ is a semicontent $R$-algebra, then $QS$ is a primary ideal of $S$ whenever  $Q$ is a primary ideal of $R$.  Conversely if $R$ is Noetherian and if $Q S$ is a prime $($resp. primary$)$ ideal of $S$ whenever $Q$ is a prime $($resp. primary$)$ ideal of $R$, then $S$ is a semicontent $R$-algebra.
\end{prop}

\section{Factor rings and local-global properties}\label{sec:localglobal}

In this section, we examine the extent to which semicontent, content, weak content, and Gaussian are local properties, and the extent to which these properties pass by base change to factor rings.

We say that the content function $\orc$ for the ring extension $R\subset S$ is {\it unital} if for all $f, g\in S$, $\orc(fg) =\orc(g)$, whenever $\orc(f) =R$.   Thus a faithfully flat ring extension $R\subset S$ is semicontent if for any multiplicatively closed subset $W$ of $R$, we have that $R_W \subset S_W$ has a unital content function.   In what follows, given $R\subset S$ and $\p \in \Spec(R)$, we let $\orc_\p$ denote the content function for the extension $R_\p\subset S_{R\setminus\p}$.   We also note that if $S$ is faithfully flat and Ohm-Rush over $R$, then for any $g\in S$, $\orc_\p(g/1) = \orc(g)R_\p$ by  Proposition~\ref{pr:omnibus}(\ref{it:OR31}).

\begin{prop} \label{localsemicontent}
Let $R\subset S$ be a faithfully flat Ohm-Rush extension.   Then the following are equivalent.
\be
\item $S$ is a semicontent $R$-algebra.
\item For every $\m\in\Max (R)$, $R_\m \subseteq S_{R\setminus \m}$ is semicontent.
\item For every $\p \in \Spec(R)$, $R_\p \subseteq S_{R\setminus \p}$ has a unital content function.
\ee
\end{prop}
\begin{proof}
(1) $\Rightarrow $ (2) is immediate from the definition.   (2) $\Rightarrow$ (3) also follows directly from the definition and the fact that localization at the prime  $\p$ is canonically  the same as first localization at a maximal ideal that contains $\p$ and then localizing that ring at the image of $\p$.

  To show (3) $\Rightarrow$ (1), let $W$ be an arbitrary multiplicatively closed subset of $R$.  Clearly condition (3) holds for all primes in $\Spec(R_W)$.   Hence, replacing $R$ with $R_W$ if need be, it suffices to show that $R\subset S$ has a unital content function.  Let $f,g \in S$ with $\orc(f) = R$.   We already know that $\orc(g) \supseteq \orc(fg)$ by definition.  On the other hand, since as noted above, the content function localizes, we have $\orc_\m(f/1) = R_\m$ for all $\m\in \Max(R)$.  Therefore by assumption $\orc(fg)_\m = \orc_\m(fg/1) = \orc_m(g/1) = \orc(g)_\m$ for all $\m\in \Max(R)$.    Hence, since the ideals are locally equal, we have $\orc(g) =\orc(fg)$.
\end{proof}

We also have analogues of the above for the weak content and Gaussian properties.
\begin{prop}\label{localweak}
Let $R \subseteq S$ be a flat Ohm-Rush extension.  The following are equivalent:
\be
\item $S$ is a weak content $R$-algebra.
\item $S_W$ is a weak content $R_W$-algebra for all multiplicative sets $W\subseteq R$.
\item $S_{R\setminus \m}$ is a weak content $R_\m$-algebra for all maximal ideals $\m$ of $R$.
\ee
\end{prop}

\begin{proof}
(1) $\implies$ (2): Let $P \in \Spec R_W$.  Then $P = \p R_W$ for some $\p \in \Spec R$ with $\p \cap W =\emptyset$.  Then by assumption, either $\p S = S$ or $\p S\in \Spec S$.  Hence we have that $P S_W = \p S_W = (\p S) S_W$ either equals $S_W$ or is an element of $\Spec(S_W)$.

(2) $\implies$ (3) is automatic.

(3) $\implies$ (1): Let $\p \in \Spec R$.  Let $f, g \in S$ such that $fg \in \p S$.  Choose a maximal ideal $\m$ of $R$ such that $\p \subseteq \m$.  Then $(f/1)(g/1) \in \p S_{R\setminus \m} \in \Spec (S_{R\setminus \m})$ (or $\p S_{R\setminus \m} = S_{R\setminus \m}$), so without loss of generality $f/1 \in \p S_{R\setminus \m}$.  Thus, there is some $a\in R\setminus \m$ with $af \in \p S$.  But then by flatness, $ac(f) = c(af) \subseteq \p$, and since $a\notin \p$ and $\p$ is prime, it follows that $c(f) \subseteq \p$, whence $f\in \p S$.  Thus, either $\p S$ is prime or $\p S = S$.
\end{proof}

\begin{prop}\label{localGauss}
Let $R \rightarrow S$ be a faithfully flat Ohm-Rush extension. The following are equivalent:
\be
\item $S$ is a Gaussian $R$-algebra.
\item $S_W$ is a Gaussian $R_W$-algebra for all multiplicative sets $W \subseteq R$.
\item $S_{R\setminus \m}$ is a Gaussian $R_\m$-algebra for all maximal ideals $\m$ of $R$.
\ee
\end{prop}

\begin{proof}
(1) $\implies$ (2) by Proposition~\ref{pr:omnibus}(\ref{it:OR31}).

(2) $\implies$ (3) is automatic.

(3) $\implies$ (1): Let $f, g \in S$.  Then by Proposition~\ref{pr:omnibus}, $\orc(fg) \subseteq \orc(f)\orc(g)$ and $\orc(fg)_\m = (\orc(f)\orc(g))_\m$ for all maximal ideals $\m$ of $R$.  Hence $\orc(fg) = \orc(f)\orc(g)$.
\end{proof}

We do not know whether a result as general as Propositions~\ref{localsemicontent}--\ref{localGauss} is possible for the property of being a content algebra, with a general base ring.   
However, we do have the following.

\begin{prop}\label{localcontent}
Let $R\subset S$ be a faithfully flat Ohm-Rush extension where $R$ is semilocal.  If $S_{R\setminus\m}$ is a content $R_\m$-algebra for each $\m \in \Max(R)$, then $S$ is a content $R$-algebra.
\end{prop}
\begin{proof}
Let $f,g \in S$ and let $\m_1,\m_2,\ldots,\m_t$ be the maximal ideals of $R$.  By assumption, for $i=1,\ldots, t $ there exists $n_i$ such that $\orc_{\m_i}(f)^{n_i}\orc_{\m_i}(g) = \orc_{\m_i}(f)^{n_i-1}\orc_{\m_i}(fg)$.  Thus if $n=\max\{n_i\}$, it follows that $\orc_{\m_i}(f)^{n}\orc_{\m_i}(g) = \orc_{\m_i}(f)^{n-1}\orc_{\m_i}(fg)$ for each $i$.  Finally observe that since $\orc(f)\supseteq \orc(fg)$ by definition, $\orc(f)^{n}\orc(g) \supseteq \orc(f)^{n-1}\orc(fg)$.  Moreover, the two ideals are locally equal, whence  the ideals are equal. This is exactly what was needed.
\end{proof}

Next we discuss what we can conclude about the ring map $R/I \rightarrow S/IS$ when $I$ is an ideal of $R$ and $R \rightarrow S$ has one of the properties we care about.  In fact, all such properties will pass to factor rings nicely.  We begin with a statement that first appears, without proof, as \cite[Remark 2.3(d)]{OhmRu-content}.  Given that we use the fact here and also crucially throughout \S3 of \cite{nmeSh-Gauss}, we provide a proof below: \begin{lemma}\label{lem:factOR}
Let $R$ be a commutative ring, let $I$ be an ideal of $R$, let $M$ be an Ohm-Rush $R$-module.  Then $M/IM$ is an Ohm-Rush $(R/I)$-module, and for any $x\in M$, we have $\orc_{R/I}(x+IM) = \orc_R(x)+I/I$.
\end{lemma}

\begin{proof}
Let $x\in M$.  It will suffice to show that $\orc_R(x)+I/I$ is the unique minimum element with respect to inclusion among ideals $K$ of $R/I$ such that $x+IM \in K\cdot(M/IM)$.  First note that $\orc_R(x)+I/I$ falls into this class, since we have $x \in \orc_R(x)M$ by assumption, so $x+IM \in (\orc_R(x)M+IM)/IM = ((\orc_R(x)+I)/I)(M/IM)$.  Next let $K$ be an ideal of $R/I$ with $x+IM\in K \cdot (M/IM)$.  There is a unique ideal $J$ of $R$ such that $J$ contains $I$ and $K=J/I$.  By assumption we have $x \in JM+IM = JM$, so $\orc_R(x) \subseteq J$, whence $\orc_R(x)+I/I \subseteq J+I/I = J/I = K$.
\end{proof}

\begin{prop}\label{pr:factorrings}
Let $R$ be a commutative ring, let $I$ be an ideal of $R$, and let $S$ be an $R$-algebra.
\begin{enumerate}
\item\label{it:factwc} if $S$ is a weak content $R$-algebra, then $S/IS$ is a weak content $(R/I)$-algebra,
\item\label{it:factsc} if $S$ is a semicontent $R$-algebra, then $S/IS$ is a semicontent $(R/I)$-algebra,
\item\label{it:factcontent} if $S$ is a content $R$-algebra, then $S/IS$ is a content $(R/I)$-algebra, and
\item\label{it:factGauss} if $S$ is a Gaussian $R$-algebra, then $S/IS$ is a Gaussian $(R/I)$-algebra.
\end{enumerate}
\end{prop}

\begin{proof}
In all cases, we have by Lemma~\ref{lem:factOR} that $S/IS$ is an Ohm-Rush $(R/I)$-algebra, with induced content function as in that lemma.

Proof of (\ref{it:factwc}): Let $f,g \in S$ and let $\p/I$ be a prime ideal of $R/I$ (i.e. $\p \in \Spec R$ with $I \subseteq \p$) that contains $\orc_{R/I}((f+IS)(g+IS)) = \orc_{R/I}(fg + IS) = \orc_R(fg) + I$.  Then $\orc_R(fg) \subseteq \p$, so since $S$ is a weak content $R$-algebra, $\p$ also contains $\orc_R(f)\orc_R(g)$.  Thus, $\p/I$ contains $\orc_R(f)\orc_R(g) + I / I = \orc_{R/I}(f+IS)\orc_{R/I}(g+IS)$.  Hence, $S/IS$ is a weak content $(R/I)$-algebra.

For the remaining cases, note that since $S$ is faithfully flat over $R$, we have that $S/IS$ is faithfully flat over $R/I$, since faithful flatness is preserved by arbitrary base change.

Proof of (\ref{it:factsc}): First we establish that when $(R,\m)$ is local and $I \subseteq \m$, the ring map $R/I \rightarrow S/IS$ has unital content function.  To see this, let $f,g\in S$ with $\orc_{R/I}(f+IS) = R/I$.  This means that $\orc_R(f)+I = R$.  But since $I \subseteq \m$ and $(R,\m)$ is local, it follows that $\orc_R(f)=R$.  Then since the map $R \rightarrow S$ has unital content function, we have $\orc_R(fg) = \orc_R(g)$.  Hence, $\orc_{R/I}((f+IS)(g+IS))= \orc_R(fg)+I/I = \orc_R(g)+I/I = \orc_{R/I}(g+IS)$.

For the general case of (\ref{it:factsc}), according to Proposition~\ref{localsemicontent}, we need only establish that for any prime ideal $P$ of $(R/I)$ (i.e. $P=\p/I$, $\p \in \Spec R$, $I \subseteq \p$), we have that the map $(R/I)_P \rightarrow (S/IS)_{(R/I)\setminus P}$ has unital content function.  But this is the same as saying that $R_\p / IR_\p \rightarrow S_\p / IS_\p$ has unital content function, which is true by what we have already shown above.

Proof of (\ref{it:factcontent}) \& (\ref{it:factGauss}):
The necessary content formulas to prove these statements follow from Lemma~\ref{lem:factOR}.
\end{proof}

\section{Noetherian rings, Artinian rings, and INC extensions}\label{sec:NoethArt}

In this section, we show that a faithfully flat weak content algebra over a Noetherian ring is a semicontent algebra.  We apply this result to algebras that arise as base change by fields, in the process showing how some subtle properties of transcendental field extensions affect the preservation of primeness of ideals.  A similar result is also shown for INC extensions of arbitrary rings.  We then specialize to the base ring being Artinian, over which any flat algebra where primes extend to primes is in fact a content algebra.

Before we move on to the results of this section, we first recall some material about \emph{strong Krull primes}.  By definition, $\p \in \Spec R$ is a \emph{strong Krull prime} of an $R$-module $L$, written $\p \in \sK_R(L)$, if for any finitely generated ideal $I$ with $I \subseteq \p$, there is some $z\in L$ with $I \subseteq \ann z \subseteq \p$. This notion serves for arbitrary commutative rings in an analogous role to the one served by associated primes for Noetherian rings.  For more information, the reader may consult \cite{IrRu-ass, nmeSh-sKflat}.  We recall the following facts:

\begin{prop}\label{pr:sKprimary}
Let $R$ be a ring, $L$ a finitely generated $R$-module, $R \rightarrow S$ a ring homomorphism, and $M$ an $R$-flat $S$-module.  \begin{enumerate}
\item\label{it:coprimary} \cite[Theorem 4]{Dut-sK} The module $L$ is coprimary (i.e. $0$ is primary in $L$) if and only if $\sK_R(L)$ is a singleton, in which case we have $\sK_R(L) = \{\sqrt{\ann L}\}$.
\item\label{it:sKeq} \cite[Theorems 4.6 and 4.7]{nmeSh-sKflat} If either $R$ is Noetherian or
the map $R \rightarrow S$ satisfies INC (i.e. its fibers are zero-dimensional),
 then \[
 \sK_S(L \otimes_RM) = \bigcup_{\p \in \sK_R(L)} \sK_S(M/\p M).
 \]
\end{enumerate}
\end{prop}

Using the above tools, we conclude the following:

\begin{thm} \label{thm:primary}
Let $\phi: R \rightarrow S$ be a flat ring homomorphism, such that either $R$ is Noetherian or $\phi$ satisfies INC.  
Let $\p \in\Spec(R)$ such that $\p S$ is a prime ideal of $S$.   Then for any $\p$-primary ideal $I$ of $R$, $IS$ is a $\p S$-primary ideal of $S$.
\end{thm}
\begin{proof}
Since $I$ is a $\p$-primary ideal, we have $\p = \sqrt{I} = \sqrt{\ann (R/I)}$.  Thus by Proposition~\ref{pr:sKprimary}(\ref{it:coprimary}), we have $\sK_R(R/I) = \{\p\}$.  In Proposition~\ref{pr:sKprimary}(\ref{it:sKeq}), let $M := S$ and $L := R/I$, which yields
 \begin{align*}
\sK_S(S/IS) &= \sK_S(R/I \otimes_R S)  = \bigcup_{\q \in \sK_R(R/I)} \sK_S(S/\q S) \\
&= \sK_S(S/\p S) = \{\p S\}.
\end{align*}
Note that the last equality follows from Proposition~\ref{pr:sKprimary}(\ref{it:coprimary}), since $\p S$ is prime and hence $\p S$-primary.  A final application of Proposition~\ref{pr:sKprimary}(\ref{it:coprimary}) then shows that $IS$ is $\p S$-primary.
\end{proof}

This generalizes \cite[Proposition 2.10]{AnHo-small}, which assumes that both $R$ and $S$ are Noetherian.

\begin{cor} \label{weaksemi}
If $S$ is a faithfully flat algebra over a Noetherian ring $R$, then $S$ is weak content over $R$ if and only if it is semicontent  over $R$.
\end{cor}

\begin{proof}
Combine Theorem~\ref{thm:primary} with Proposition~\ref{pr:primaries}.
\end{proof}

Recall the following theorem from Ananyan-Hochster, slightly rephrased:

\begin{thm}\label{thm:AH-prseq}\cite[Corollary 2.9]{AnHo-small}
Let $K$ be an algebraically closed field, and let $L$ be an extension field of $K$.  Set $S := L[x_1, \ldots, x_n]$.  Let $g_1, \ldots, g_s \in S$ be a sequence of homogeneous nonconstant polynomials that are  algebraically independent over $K$, such that every initial sequence generates a prime ideal of $S$.  Then in the inclusion $R := K[g_1, \ldots, g_s] \subseteq S$, prime ideals of $R$ extend to prime ideals of $S$.
\end{thm}

Reinterpreted in the language of Ohm-Rush content, we have:

\begin{cor}\label{cor:AH}
Let $K$, $L$, $\mathbf x$, $\mathbf g$, $R$ and $S$ be as in Theorem~\ref{thm:AH-prseq}. Then $S$ is a semicontent $R$-algebra.
\end{cor}

\begin{proof}
The argument in the paragraph \cite[``Extensions of prime ideals'']{AnHo-small} shows that $S$ is free as an $R$-module, hence it is flat and Ohm-Rush.  Then since primes extend to primes (and hence we have a faithfully flat weak content algebra), an appeal to Corollary~\ref{weaksemi} or \cite[Proposition 2.10]{AnHo-small} finishes the proof.
\end{proof}

In particular, we may consider the special case where $\mathbf g = \mathbf x$, to obtain the following:

\begin{cor}\label{cor:aclosed}
Let $L/K$ be a field extension such that $K$ is algebraically closed.  Then the ring extension $K[x_1, \ldots, x_m] \rightarrow L[x_1, \ldots, x_n]$ yields a semicontent algebra for any pair of nonnegative integers $m\leq n$.
\end{cor}

It is natural to ask if there are situations where the base field is \emph{not} algebraically closed but we still have a semicontent extension.  In Proposition~\ref{pr:irredfield}, we will see that we must at least assume that $K$ is (relatively) algebraically closed in $L$.  It turns out that in the two-variable case, even this is not enough (see Example~\ref{ex:aclosed}).

To analyze the situation properly, we recall the following notions and results from the theory of transcendental field extensions:

\begin{defn}
Let $L/K$ be a field extension.  We call the extension \begin{itemize}
\item \emph{separable} if there is a transcendence basis $T$ for $L$ over $K$ such that $L/K(T)$ is a separable algebraic extension,
\item \emph{primary} if the biggest separable algebraic extension of $K$ within $L$ is $K$ itself,
\item \emph{algebraically closed} if the only elements of $L$ that are algebraic over $K$ are already in $K$.
\end{itemize}
\end{defn}

Note the following facts: \begin{itemize}
\item If a field extension $L/K$ is separable and primary, the extension is algebraically closed \cite[Lemma 2.6.4 and Corollary 2.6.14(d)]{FrJa-fieldbook}.
\item The converse holds if $K$ is perfect (e.g. if the fields have characteristic $0$ or $K$ is finite), as every extension of a perfect field is separable.
\end{itemize}

\begin{prop}\label{pr:regext}
Let $L/K$ be a separable, primary extension of fields.  Let $x_1, \ldots, x_n$ be indeterminates over $L$ and choose $0\leq m\leq n$.  Then the extension $R:= K[x_1, \dots, x_m] \subseteq L[x_1, \ldots, x_n] =: S$ is semicontent.
\end{prop}

\begin{proof}
Write $T := L[x_1, \ldots, x_m]$.  Since $S$ is a content $T$-algebra, by \cite[Theorem 3.10]{nmeSh-OR} it will suffice to show that $T$ is a semicontent $R$-algebra.

For the usual reasons, $T$ is faithfully flat and module-free (hence Ohm-Rush) over $R$.  Hence by Corollary~\ref{weaksemi}, it is enough to show that prime ideals of $R$ extend to prime ideals of $T$.  Accordingly, let $\p \in \Spec R$.  Set $A := R/\p$.  Note that $T/\p T = A \otimes_K L$.  By Lemma~\ref{lem:regext} below, $T/\p T$ is therefore an integral domain.  Hence $\p T \in \Spec T$, completing the proof.
\end{proof}

It remains to prove the following presumably well-known Lemma.

\begin{lemma}\label{lem:regext}
Let $L/K$ be a separable, primary field extension, and let $A$ be a $K$-algebra.  If $A$ is an integral domain, so is $A \otimes_KL$.
\end{lemma}

\begin{proof}
Let $F$ be the fraction field of $A$.  By \cite[4.3.2 and 4.3.5]{EGA-4.2}, $F \otimes_K L$ is reduced and has irreducible prime spectrum.  Hence it is an integral domain.  But since $L$ is a flat $K$-algebra, the injective map $A \hookrightarrow F$ base-changes to an injective map $A \otimes_K L \hookrightarrow F \otimes_K L$.  Since any subring of an integral domain is an integral domain, we are done.
\end{proof}

The combination of Corollary~\ref{cor:aclosed} and Proposition~\ref{pr:regext} leads naturally to the question: Is it true that for \emph{any} field extension $L/K$ with $K$ algebraically closed in $L$ and any $n\in \N$, $L[x_1, \ldots, x_n]$ is a content algebra over $K[x_1, \ldots, x_n]$?  We will see in Proposition~\ref{pr:irredfield} that when $n=1$, the answer is ``yes''.  However, the answer can be ``no'' when $n=2$.

\begin{example}\label{ex:aclosed}
Let $p$ be a prime number, $\F_p$ the field of $p$ elements, let $a,b,s,t,x,y$ be algebraically independent indeterminates over $\F_p$ and set $K=\F_p(s,t)$.  Set $C := K[a,b] / (sa^p + tb^p - 1)$.  Note that $C$ is an integral domain (since $sa^p + tb^p - 1$ is irreducible).  Let $L$ be the fraction field of $C$.  Then $K$ is algebraically closed in $L$ \cite[p. 384]{Mac-modI}.  Set $R := K[x,y]$ and $S:= L[x,y]$.  Write $\m := (x^p - s, y^p - t)R$.  Then $\m$ is a maximal ideal of $R$, because $R/\m = \F_p(s^{1/p}, t^{1/p})$ (essentially $x$ and $y$ act as $p$th roots of $s$, $t$ respectively).  However, $\m S$ is not even a radical ideal of $S$, because $(xa+yb-1)^p \in \m S$ even though $xa+yb-1 \notin \m S$.

In particular, even though $K$ is algebraically closed in $L$, $L[x,y]$ is \emph{not} a weak content $K[x,y]$-algebra.
\end{example}

Finally in this vein, we have the following (potentially) stronger result in comparison to Lemma~\ref{lem:regext} and Proposition~\ref{pr:regext}.  (It is only really stronger if content algebras are different from semicontent algebras).

\begin{prop}\label{pr:purecontent}
Let $L/K$ be a purely transcendental field extension, and let $A$ be a ring that contains $K$.  Then $A \otimes_K L$ is a \emph{content} algebra over $A$.
\end{prop}

\begin{proof}
Let $T$ be a transcendence basis of $L$ over $K$.  Set $B : = A[T]$, and let $W := K[T] \setminus \{0\}$.  Then $W$ is a multiplicatively closed subset of $B$, and in the content map from $B$ to the ideals of $A$, every element of $W$ has unit content.  Thus by Proposition~\ref{pr:omnibus}(\ref{it:OR62}), $B_W = A \otimes_KL$ is a content $A$-algebra.
\end{proof}

Recall that in Theorem~\ref{thm:primary}, we assumed that either the base ring is Noetherian or the ring map satisfies INC.  By Corollary~\ref{weaksemi}, the former condition makes weak content faithfully flat algebras semicontent.  It turns out that the latter one does as well.

\begin{lemma}\label{lem:INCinj}
Let $R\subset S$ satisfy INC and suppose that prime ideals extend 
to prime ideals.  Then the induced map $\Spec S \rightarrow \Spec R$ is injective, and for any $\p$ in the image of the map, its unique preimage is $\p S$.
\end{lemma}

\begin{proof}
Let $P\in \Spec S$.   Set $\p := P\cap R$.   Then $\p \subseteq \p S \cap R \subseteq P \cap R = \p$, so both of the prime ideals in the chain $\p S \subseteq P$ contract to $\p$.  Hence by INC, $P = \p S$, and we have injectivity.
\end{proof}

\begin{prop}\label{pr:weaksemiINC}
Let $S$ be a faithfully flat weak content $R$-algebra such that $R \subset S$ satisfies INC.  Then $S$ is semicontent over $R$.
\end{prop}

\begin{proof}
By Proposition~\ref{localsemicontent}, we need only show that the extension $R_\p 
\subset S_{R\setminus\p}$ has unital content.  Hence, we may assume that $R$ is local with 
maximal ideal $\p$.   First we claim that $S$ is local, with maximal ideal $\p S$.  To see this, let $\n$ be a maximal ideal of $S$.  Since $(\n \cap R)S \cap R = \n \cap R$ by purity, and since $(\n \cap R)S$ is prime, we have $(\n \cap R)S = \n$ by the INC property.  But  $\n\cap R \subseteq \p$, as $\p$ is the unique maximal ideal of $R$.  Thus, $\n = 
(\n\cap R)S \subseteq \p S \subseteq \n$, whence $\n = \p S$.

Now let $f, g \in S$ with $\orc(f) = R$.   Then $f \not\in\p S$.  Hence $f$ is a unit of 
$S$ and so $\orc(fg) = \orc(g)$.
\end{proof}

We shift our attention next to direct products of rings and Artinian rings.

\begin{prop} \label{directsum}
Let $R\subset S$ be rings and suppose that $R$ decomposes as $R=R_1\times R_2\times \cdots\times R_n$.   Then $S = S_1\times\cdots S_n$, where $R_i\subset S_i$.  Moreover, $S$ is $($semi$)$content over $R$ if and only if $S_i$ is $($semi$)$content over $R_i$ for each $i=1,2,\ldots, n$.
\end{prop}

This result can be seen as complementary to \cite[Corollary 1.4]{OhmRu-content}
\begin{proof}
 Let $e_1,e_2,\ldots,e_n$ be the corresponding orthogonal set of idempotents for the decomposition of $R$ (so $R_i = e_iR$).  Let $S_i = e_iS$ and so $S_i$ is an algebra over $R_i$ and $S=S_1\times S_2\times \cdots\times S_n$.
Any ideal $J$ of $R$ has the form $J_1\times J_2\times \cdots \times J_n$, where $J_i = e_iJ$.   Hence it is clear that $S$ is faithfully flat over $R$ if and only if each $S_i$ is faithfully flat over $R_i$.   Note  that for $J$ an ideal of $R$ and $f\in S$, then $f\in JS$ if and only if $e_i f =f_i \in J_iS$.  Since  intersection distributes over the direct product,
we also have for $f\in S$, that  $\orc(f) = \sum \orc_i(f_i)$, where $\orc_i$ denotes the Ohm-Rush content function of  $S_i$ over $R_i$.  Then $f\in \orc(f)S$, if and only if $f_i\in \orc_i(f_i)$ for $i=1,2,\ldots, n$.   Hence $S$ is Ohm-Rush over $R$ if and only if $S_i$ is Ohm-Rush over each $R_i$.

Let $f,g \in S$ with $\orc(f)^{n}\orc(g) = \orc(f)^{n-1}\orc(fg)$ for some $n$.  It follows that for each $i=1,2,\ldots,n$, $\orc(f_i)^{n}\orc(g_i) = \orc(f_i)^{n-1}\orc(f_ig_i)$.   Conversely suppose that for each $i$, there exists $n_i$ such that $\orc(f_i)^{n_i}\orc(g_i) = \orc(f_i)^{n_i-1}\orc(f_ig_i)$.  If we let $n=\max\{n_i\}$, it is clear that $\orc(f)^{n}\orc(g) = \orc(f)^{n-1}\orc(fg)$.  Hence $S$ is content over $R$ if and only if $S_i$ is content over $R_i$ for each $i$.

Next we show that being semicontent is also a coordinate-wise property.  First observe that an ideal $\m$ of $R$ is a maximal ideal if and only if for some $i=1,2,\ldots, n$, $\m$ has the form $R_1\times R_2\times \cdots \times\m_i\times \cdots \times R_n$, where $\m_i$ is a maximal ideal of $R_i$.    It then follows that $R_\m = (R_i)_{\m_i}$ and $S_{R\setminus \m} = (S_i)_{R_i\setminus\m_i}$ canonically.   Hence by Proposition~\ref{localsemicontent} $S$ is semicontent over $R$ if and only if $S_i$ is semicontent over $R_i$ for each $i=1,2,\ldots,n$.
\end{proof}

\begin{prop}\label{pr:ArtOR}
Let $R$ be an Artinian ring and let $M$ be a flat $R$-module.  Then $M$ is an Ohm-Rush $R$-module.
\end{prop}

\begin{proof}
Let $g \in M$, and let $X := \{I $ ideal of $R \mid g \in IM\}$.  Note that $X$ is closed under finite intersection, since by flatness of $M$ over $R$, we have $IM \cap JM = (I \cap J)M$ for all ideals $I$, $J$.  On the other hand, since $R$ is Artinian, it satisfies the minimality condition.  That is, every nonempty set of ideals has a minimal element.  We have $(1) \in X$, so $X \neq \emptyset$.  Accordingly, let $J$ be a minimal element of $X$.  Then $g\in JM$, and for any $I \in X$, we have $I \cap J \in X$, whence by minimality $I \cap J = J$, which means that $J \subseteq I$.  Thus by definition, $J = \orc(g)$, so that $M$ is an Ohm-Rush $R$-module.
\end{proof}

\begin{thm}\label{thm:Artequiv}
Let $R\subset S$ where $R$ is Artinian.  Then the following are equivalent.
\begin{enumerate}
\item $S$ is a content algebra over $R$.
\item $S$ is a semicontent algebra over $R$.
\item $S$ is faithfully flat over $R$ and a weak content algebra over $R$.
\item $S$ is flat over $R$ and $\p S \in \Spec S$ for all $\p \in \Spec R$.
\end{enumerate}
\end{thm}
\begin{proof} We already know that (1) $\Rightarrow$ (2) $\Rightarrow$ (3).  (3) $\Rightarrow$ (4) by Proposition~\ref{pr:omnibus}(\ref{it:Ruprimes}).   Since $R$ is Artinian, (4) $\Rightarrow$ (3) by Proposition~\ref{pr:ArtOR} and Proposition~\ref{pr:omnibus}(\ref{it:Ruprimes}), and because prime ideals do not extend to the whole ring.  Since $R$ is Noetherian, (3) $\Rightarrow$ (2) by Corollary~\ref{weaksemi}.  For the remaining implication (2) $\Rightarrow$ (1), since an Artinian ring is a finite direct product of local Artinian rings,  we may assume that $R$ is local with maximal ideal $\m$  by Proposition~\ref{directsum}.

Let $f,g \in S$.  If $\orc(f) \subseteq \m$, then for some $n$ we have $\orc(f)^n =0$, so that $\orc(f)^{n+1} \orc(g) = \orc(f)^n \orc(fg)$ automatically.  Otherwise $\orc(f)=R$.   In this case, the Dedekind-Mertens equation reduces to  the equation $\orc(fg) = \orc(g)$.  However, this is just the semicontent condition, which is what we are assuming.
\end{proof}

\begin{example}\label{ex:Art}
It is natural at this point to ask whether over an Artinian ring, content algebras are the same as Gaussian algebras.  In general, they are not.

For instance, let $R=k[a,b] / (a^2,b^2)$, where $k$ is any field and $a, b$ are indeterminates, and let $S=R[x]$. Then $S$ is a content $R$-algebra because it is a polynomial extension. However, $\orc(ax+b)\orc(ax-b)= (a,b)^2 = abR \neq 0$, but $\orc((ax+b)(ax-b)) = \orc(a^2 x^2 - b^2) = \orc(0)=0$, so $S$ is \emph{not} a Gaussian $R$-algebra.
\end{example}

\section{Extensions of valuation rings}\label{sec:valbase}

The general question we pose in this section is given a valuation ring $V$ and a $V$-algebra $S$, when is $S$ a content algebra over $V$?  We first note some basic facts.  As long as $S$ is a domain, it is torsion free as a $V$-module, whence it is a priori flat over $V$.  More generally, we recall the following result, which is \cite[Proposition 4.7]{nmeSh-OR}:

\begin{prop}\label{pr:ES47} 
Let $R$ be a Pr\"ufer domain, and let $S$ be a faithfully flat Ohm-Rush $R$-algebra.  The following are equivalent: \begin{enumerate}
\item $S$ is weak content over $R$.
\item $S$ is semicontent over $R$.
\item $S$ is content over $R$.
\item $S$ is Gaussian over $R$.
\item For all maximal ideals $\m$ of $R$, we have $\m S \in \Spec S$.
\end{enumerate}
\end{prop}

For this reason, when $R$ is a Pr\"ufer (e.g. valuation) domain and $S$ is a faithfully flat Ohm-Rush $R$-algebra, it does not matter whether we speak of Gaussian algebras, content algebras, semicontent algebras, weak content algebras, etc.  They all coincide.  In the sequel we consistently use the term ``content algebra'' only because it is the (chronologically) oldest such term.

 It will be convenient  to introduce some notation that will simplify our discussion.   For a ring $V$, let
$\mathfrak{K}_V:=\{ I$ an ideal of $V$: $I$ is the intersection of all the ideals that properly contain $I \}$. We will suppress the subscript if there is no chance for confusion. For each $I$ an ideal of $V$, let $\mathcal{K}_I:= \{ J\subset V : J$ properly contains $I\}$.   Hence $I\in\mathfrak{K}\Leftrightarrow I=\bigcap_{J\in \mathcal{K}_I} J$.  Finally, when $S$ is a $V$-algebra and $a\in S$, we set $\mathcal{L}_a:= \{I\subset V: a\in IS\}$.   Hence $S$ is an Ohm-Rush algebra over $V$ if and only if for all $a\in S, \orc(a) \in \mathcal{L}_a$ (i.e. $\mathcal{L}_a$ has a least element).

\begin{prop}\label{pr:valOR}
Let $V \subset S$, where $V$ is a valuation ring and $S$ is integral domain.  Then $S$ is an Ohm-Rush algebra over $V$ if and only for every $I\in \mathfrak{K}$, $IS =\bigcap_{J\in \mathcal{K}_I} JS$.

\end{prop}
\begin{proof}
First assume that  for each $I\in \frk, IS =\bigcap_{J\in \mathcal{K}_I} JS$.   We want to show that if $a\in S$, then $\orc(a)\in \mathcal{L}_a$.  We divide into two cases.  First suppose that $\orc(a) \not\in \frk$.  Thus $\orc(a)$ is not the intersection of the ideals that strictly contain it.   Hence by definition, $\orc(a) \in \mathcal{L}_a$.

In the second case, $\orc(a) \in \frk$, so $\orc(a) = \bigcap_{J\in \mathcal{K}_{\orc(a)}} J$.   Furthermore by assumption
 $$\orc(a)S = \Bigg(\bigcap_{J\in \mathcal{K}_{\orc(a)}} J\Bigg)S = \bigcap_{J\in \mathcal{K}_{\orc(a)}} (JS).$$
   Now suppose that $a\not\in \orc(a)S$.   Then $a\not\in JS$ for some $J \in \mathcal{K}_{\orc(a)}$.   As $V$ is a valuation ring, the ideals of $V$ are linearly ordered.  Thus  all $K\in \mathcal{L}_a$ contain this $J$.   In particular $\bigcap_{K\in \mathcal{L}_a}K \supseteq J \supset \orc(a)$, a contradiction.

   The converse follows immediately from Proposition~\ref{pr:omnibus}(\ref{it:OR12ii}).
\end{proof}

We next show that heights are preserved in content extensions of Pr\"ufer domains.

\begin{lemma}\label{lem:extended}
Let $R$ be a B\'ezout domain (e.g. a valuation ring) and let $S$ be a flat Ohm-Rush algebra over $R$. Let $J$ be an ideal of $R$, let $Q$ a prime ideal of $S$, and suppose $Q \subseteq JS$.  Then $Q$ is an extended ideal.  That is, $Q = (Q \cap R)S$.
\end{lemma}

\begin{proof}
Let $g\in Q$.  Since $R$ is B\'ezout and content ideals are finitely generated, $\orc(g)$ is principal.  That is, $\orc(g) = rR$ for some $r\in R$.  Since $g\in \orc(g)S$, it follows that there is some $s\in S$ with $g=rs$.  We have \[
rR = \orc(g) = \orc(rs) =r\orc(s),
\]
with the last equality following from Proposition~\ref{pr:omnibus}(\ref{it:ORflat}).  Then by cancellation, $R=\orc(s)$.  In particular (again by the Ohm-Rush property), it follows that $s\notin JS$, so $s\notin Q$.  But $g=rs\in Q$, so $r\in Q \cap R$.  Therefore, $g=rs \in (Q\cap R)S$.
\end{proof}

\begin{thm}\label{thm:height}
Let $R \rightarrow S$ be a content algebra map, where $R$ is a Pr\"ufer domain.  Then for any $\p \in \Spec R$, $\hgt \p = \hgt \p S$.
\end{thm}

\begin{proof}
Since $\hgt \p R_\p = \hgt \p$, $\hgt \p S_{R \setminus \p} = \hgt \p S$, and $R_\p \rightarrow S_{R \setminus \p}$ is a content algebra map (by Proposition~\ref{pr:omnibus}(\ref{it:OR62})), we may assume that $R$ is a valuation ring.

Since $S$ is faithfully flat over $R$ and prime ideals extend to primes of $S$, we have that \emph{distinct} prime ideals of $R$ extend to distinct primes of $S$.  Hence any chain in $R$ extends to a chain in $S$ of the same length, which then implies that $\hgt \p \leq \hgt \p S$.

Conversely, let $0=Q_0 \subsetneq \cdots \subsetneq Q_t = \p S$ be a chain of prime ideals in $S$.  Then by Lemma~\ref{lem:extended}, each $Q_j = \p_j S$ for some prime ideal $\p_j \in \Spec R$.  In particular, we have $0=\q_0 \subsetneq \cdots \subsetneq \q_t = \p$, so that $\hgt \p \geq t$.  Since the chain in $\Spec S$ under $\p S$ was arbitrary, it follows that $\hgt \p \geq \hgt \p S$.  Thus $\hgt \p S = \hgt \p$.
\end{proof}

\begin{cor}\label{cor:dimcatenary}
Let $R \rightarrow S$ be a content algebra map, where $R$ is a Pr\"ufer domain and $S$ is catenary.  Then for any $P \in \Spec S$, we have \[
\dim S_P = \dim R_\p + \dim S_P / \p S_P,
\]
where $\p = P \cap R$.
\end{cor}

\begin{proof}
Since $0S \subseteq \p S \subseteq P$ are prime ideals of $S$, the catenarity assumption implies that \[
\hgt P = \hgt \p S + \hgt (P/\p S).
\]
But by Theorem~\ref{thm:height}, $\hgt \p = \hgt \p S$.  Thus, \[
\dim S_P = \hgt P = \hgt \p S + \hgt P/\p S = \hgt \p +\hgt P/\p S = \dim R_\p + \dim S_P / \p S_P.
\]
\end{proof}

\begin{rem}
The above corollary is an analogue of what happens with arbitrary flat maps between Noetherian rings \cite[Theorem 15.1 (ii)]{Mats}.  It is also a generalization of what happens with a \emph{polynomial} extension of a finite rank valuation domain, by Seidenberg's theorem \cite[Theorem 4]{Sei-dim2}, since Pr\"ufer domains locally of finite dimension are universally catenary \cite[Corollary 3.9]{MaMo-strongS}.
\end{rem}

Without the catenarity assumption, we get the following partial result.

\begin{prop}
Let $V \rightarrow S$ be a content algebra map, where $V$ is a valuation domain, and where both $S$ and $V$ are finite-dimensional.  Let $0 = \p_0 \subsetneq \cdots \subsetneq \p_d$ be the spectrum of $V$.  For each $0 \leq i \leq d$, set $P_i := \p_iS$, and set $t_i := \max\{1, \dim(S/P_i)_{V \setminus \p_i}\}$.  Then $\dim S \leq \sum_{i=0}^d t_i$.
\end{prop}

\begin{proof}
We proceed by induction on $d$, the dimension of $V$.  When $d=0$, this reduces to the trivial statement that $\dim S \leq \dim S$.

Now suppose $d>0$ and that the statement holds for valuation domains of smaller dimension.  Let $Q_m \supsetneq Q_{m-1} \supsetneq \cdots \supsetneq Q_1 \supsetneq Q_0=0$ be a descending chain of prime ideals of $S$.  Now suppose some nonzero $Q_j$ is an extended prime ideal; say $Q_j = \p_iS$.  Then by Theorem~\ref{thm:height}, $\hgt Q_j = i$, and by induction, we have \[
m \leq \hgt Q_j + \dim S/Q_j = i + \dim S/Q_j \leq i + \sum_{\alpha=i}^d t_\alpha \leq \sum_{\alpha=0}^d t_\alpha.
\]
On the other hand, suppose no nonzero $Q_j$ is extended.  Then at most $t_0+1$ of the $Q_j$ contract to the zero ideal of $V$, and at most $t_i$ of the $Q_j$ contract to $\p_i$.  Since there are $m+1$ of the $Q_j$'s, we have $m+1 \leq 1+\sum_{i=0}^d t_i$, completing the proof.
\end{proof}

\section{Content algebras where both rings are valuation}\label{sec:valval}
\begin{lemma}\label{lem:nonprincipal}
Let $(V,\m)$ be a valuation ring such that $\m$ is not principal.  Then there is some nonzero $x\in \m$ such that $(x) = \bigcap \mathcal K_{(x)}$.
\end{lemma}

\begin{proof}
Either $\m$ is \emph{branched} (in which case there is some prime ideal $\p$ such that $\dim V/\p = 1$) or \emph{unbranched} (in which case $\m$ is not minimal over any principal ideal) \cite[Theorem 17.3 (e)]{Gil-MIT}.

\noindent \textbf{Case 1:} Assume $\m$ is branched.  Then let $\p \subset \m$ with $\dim V/\p = 1$.  Let $x\in \m \setminus \p$.  Since $\p \subsetneq (x)$, we may pass to $V/\p$ and assume that $\dim V = 1$.  Since $\m$ is not principal, the value group $G$ of $V$ is then a non-discrete subgroup of $\R$.  Choose $\alpha>0$ with $\alpha \in G$.  Then there is some increasing sequence $\{\alpha_i\}_{i \in\N}$ with $0<\alpha_i \in G$ such that $\displaystyle \lim_{i \rightarrow \infty} \alpha_i = \alpha$.

Choose $x\in \m$ with $\nu(x) = \alpha$ and $x_i \in \m$ with $\nu(x_i) = \alpha_i$.  Then for each $i$, we have $(x) \subsetneq (x_i)$; hence $(x_i) \in \mathcal K_{(x)}$.  On the other hand, choose $y\in \bigcap \mathcal K_{(x)}$.  Then for each $i$, we have $y \in (x_i)$, so $\nu(y) \geq \alpha_i$.  By the limit statement, it then follows that $\nu(y) \geq \alpha = \nu(x)$, so that $y \in xV$.  This completes the proof that if $\m$ is branched, we can find $x$ with $(x) = \bigcap \mathcal K_{(x)}$.

\noindent\textbf{Case 2:} Hence we may assume $\m$ is unbranched.  Take any $0 \neq x \in \m$.  Let $y \in \m$ with $y \notin (x)$.  Then $(x) \subsetneq (y)$, so $x \in \m y$, whence $x/y \in \m$.  Let $\p$ be minimal over $(x/y)$.  Since $\m$ is unbranched, we have $\m \neq \p$.  Choose $t\in \m \setminus \p$, and write $z := x/t$; we have $x \in (x/y) \subseteq \p \subseteq (t)$, so $z=x/t \in V$.

Now, $x=tz \subseteq \m z$, so $(x) \subsetneq (z)$.  Moreover, $(x/y) \subseteq \p \subsetneq (t)$, so $x/y \in \m t$, whence $z=x/t \in \m y$, so $(z) \subsetneq (y)$.  Thus, $(z) \in \mathcal K_{(x)}$ and $y \notin (z)$, so $y \notin \bigcap \mathcal K_{(x)}$.  Since $y$ was chosen to be an arbitrary element of $\m \setminus (x)$, it follows that $(x) =  \bigcap \mathcal K_{(x)}$.
\end{proof}

\begin{prop}\label{pr:maxextend}
Let $(V,\m) \rightarrow (S,\n)$ be a faithfully flat map of valuation rings, where $S$ is Ohm-Rush over $R$ and $\m$ is not principal.  Then $\m S = \n$.
\end{prop}

\begin{proof}
By Lemma~\ref{lem:nonprincipal}, we may choose $0\neq x \in \m$ such that $(x) = \bigcap \mathcal K_{(x)}$.  If $\m S \neq \n$, then let $y \in \n \setminus \m S$.  We have $x/y \in \n$.  We claim that $\mathcal L_{x/y} = \mathcal K_{(x)}$.

To see this, let $J \in \mathcal L_{x/y}$.  That is, $J$ is an ideal of $V$ with $x/y \in JS$.  Since $y\in \n$, $J \neq (x)$.  On the other hand, $x\in yJS \cap V \subseteq JS \cap V = J$ by purity.  Hence $J \in \mathcal K_{(x)}$.

Conversely, let $J \in \mathcal K_{(x)}$.  That is, $(x) \subsetneq J$.  Let $j \in J \setminus (x)$.  Then $xV \subsetneq jV$, so $x/j \in \m$.  Then $x=jt$, where $t \in \m \subseteq \m S \subseteq yS$; say $t=ys$, $s\in S$.  Then $x=jt = yjs$, so $x/y =js \in JS$.  Thus, $J \in \mathcal L_{x/y}$.

It follows that $\orc(x/y) = \bigcap \mathcal L_{x/y} = \bigcap \mathcal K_{(x)} = (x)$.  By the Ohm-Rush property, then, $x/y \in \orc(x/y)S = xS$, so there is some $u\in S$ with $x/y = xu$, so $x=xyu$, whence $yu=1$, contradicting the fact that $y$ is not a unit.
\end{proof}

\begin{example}\label{ex:nthroot}
Note that the assumption in the above proposition that $\m$ not be principal is necessary.  For an example, let $V = k[\![x]\!]$, $k$ any field, let $n \in \N$ with $n \geq 2$, and let $S = k[\![x^{1/n}]\!]$.  Then $S$ is free of rank $n$ over $V$, hence Ohm-Rush.  It is faithfully flat as it is an inclusion of valuation domains and the maximal ideal of $V$ does not extend to all of $S$.  But $\m S = xS$ is not prime, because $(x^{1/n})^n \in \m S$ even though $x^{1/n} \notin \m S$.
\end{example}

The next lemma is surely known, but we include it for completeness
\begin{lemma}\label{lem:valdescent}
Let $R \rightarrow S$ be a faithfully flat ring map.  If $S$ is a valuation ring (resp. a Pr\"ufer domain), then so is $R$.
\end{lemma}

\begin{proof}
First suppose $S$ is a valuation domain.  Let $I, J$ be ideals of $R$ with $I \neq J$ and $I \nsubseteq J$.  Then by faithful flatness, $IS \nsubseteq JS$ and $IS \neq JS$, so since $S$ is a valuation domain, $JS \subseteq IS$.  Thus by purity of the faithfully flat extension, $J \subseteq I$.  Hence, the ideals of $R$ are linearly ordered, so $R$ is a valuation ring.

If $S$ is a Pr\"ufer domain, let $\m$ be a maximal ideal of $R$.  Since $\m S \neq S$ (by faithful flatness), there is some maximal ideal $\n$ of $S$ with $\n \cap R = \m$.  Then the induced map $R_\m \rightarrow S_\n$ is faithfully flat, so that since $S_\n$ is a valuation ring, the first part of the proof shows that $R_\m$ is also a valuation ring.  Since the map $R \rightarrow S$ is injective (by faithful flatness), $R$ is an integral domain, which moreover is locally valuation.  Hence it is a Pr\"ufer domain.
\end{proof}

Note that in Example~\ref{ex:nthroot}, $S$ is not a content $V$-algebra.  However, if we add the assumption, then we get the following:

\begin{cor}\label{cor:maxextend}
Let $R \rightarrow S$ be a content algebra map, where $S$ is a Pr\"ufer domain and $R$ is not a field.  Then for any maximal ideal $\m$ of $R$, $\m S$ is a maximal ideal of $S$.
\end{cor}

\begin{proof}
First assume that $(S,\n)$ is local.  Then by Lemma~\ref{lem:valdescent}, $R$ is a valuation domain.  If $\m$ is principal, then since $\m S$ is a nonzero prime principal ideal of the valuation ring $S$, it must be the maximal ideal of $S$ by \cite[Theorem 17.3(a)]{Gil-MIT}.  On the other hand, if $\m$ is not principal, then $\m S = \n$ by Proposition~\ref{pr:maxextend}.

For the general case, since $S$ is faithfully flat over $R$, there is some maximal ideal $\n$ of $S$ with $\m S \subseteq \n$.  Then $R_\m \rightarrow S_\n$ is a content algebra map with $S_\n$ a valuation ring, so by the previous paragraph we have $\m S_\n = \n S_\n$. But then since $\m S$ is prime, we have $\m S = \n$.
\end{proof}

For examples where Corollary~\ref{cor:maxextend} applies, see Proposition~\ref{pr:irredfield} with $n=1$, as well as Theorem~\ref{thm:valgroups}.

\begin{prop}\label{pr:cval}
Let $(R,\m) \rightarrow (S,\n)$ be a content algebra, where $S$ is a valuation domain.  Then for any $g\in S$, we have $\orc(g) = gS \cap R$.
\end{prop}

\begin{proof}
If $g \in S \setminus \n$, then $g$ is a unit, so $\orc(g) = R = S \cap R = gS \cap R$.  So we may assume $g\in \n$.  But $\n = \m S$, so $\orc(g) \subseteq \m$.  Since $R$ is a valuation ring by Lemma~\ref{lem:valdescent}, and since $\orc(g)$ is a finitely generated ideal of $R$, it must be principal. Say $\orc(g) = rR$.  Then $g \in \orc(g) S = rS$, so there is some $s\in S$ with $g=rs$.  But then \[
rR = \orc(g) = \orc(rs) = r\orc(s),
\]
where the last equality is by Proposition~\ref{pr:omnibus}(\ref{it:ORflat}) so by cancellation, $\orc(s) = R$.  If $s\in \n=\m S$, we would have $\orc(s) \subseteq \m$.  Hence $s$ is a unit of $S$.  Thus, $r = gs^{-1} \in gS \cap R$, so $g \in rV \subseteq (gS \cap R)S$.  Now let $J$ be any ideal with $g \in JS$.  Then $gS \cap R \subseteq JS \cap R = J$ by purity.  Hence $\orc(g) = gS \cap R$.
\end{proof}

\begin{example}\label{ex:polynomial}
The above property is very special to valuation domains; it fails in general even for non-local principal ideal domains.

For instance, Let $L/K$ be an algebraically closed field extension (i.e. the only elements of $L$ algebraic over $K$ are already in $K$), $R := K[x]$, $S := L[x]$, and let $\orc = \orc_{SR}$ be the Ohm-Rush content function.  Since $L$ is free as a $K$-module, it follows that $S$ is free (and hence Ohm-Rush) over $R$ via the same generators.  Let $g$ be a monic irreducible polynomial in $L[x]$.  Let $I$ be an ideal of $K[x]$ with $g \in IS$. Since $R$ is a principal ideal domain, we have $I=(f)$ for some $f\in K[x]$. Also, $g\in fS$.  But since $g$ is irreducible, it follows that $f$ is a unit of $S$, whence $f \in K^\times$.  That is, $I = R$.  Since $I$ was arbitrary, we have $\orc(g) = R$.  On the other hand, $gS \cap R = 0$.  To see this, suppose $h \in (gS \cap R)\setminus \{0\}$.  Let $h=h_1 \cdots h_t$ be the unique factorization in $R$ into monic irreducibles.  Since irreducibles in $R$ remain irreducible in $S$ by Proposition~\ref{pr:irredfield}, this is also the unique factorization into monic irreducibles in $S$.  But then for some $1\leq j \leq t$, we have $g=h_j$, contradicting the fact that $g\notin R$.

This becomes a limiting example for Proposition~\ref{pr:cval} because of the following result.
\end{example}

\begin{prop}\label{pr:irredfield}
Let $L/K$ be a field extension, let $n$ be a positive integer, let $R := K[x_1]$, $S:= L[x_1, \ldots, x_n]$, and let $R \rightarrow S$ be the natural ring extension.  Then $S$ is free as an $R$-module, but it is a content algebra if and only if $K$ is algebraically closed in $L$.
\end{prop}

\begin{proof}
Since $B := L[x_2, \ldots, x_n]$ is free as a $K$-module, the same holds for $S = B \otimes_K R$ over $R$.  In particular, $S$ is Ohm-Rush and flat over $R$.

Suppose $K$ is not algebraically closed in $L$.  Write $x=x_1$.  Let $\alpha\in L$ be algebraic over $K$.  Then there is some irreducible polynomial $f\in K[x]$ of degree at least two such that $f(\alpha)=0$.  In particular, $f$ factors in $L[x]$ as $f(x) = (x-\alpha)g(x)$, $g\in L[x]$.  Let $\p = fR$.  Then $\p$ is prime, but $\p S$ is not prime because $(x-\alpha)g \in \p S$, $x-\alpha \notin \p S$, and $g \notin \p S$ (e.g. for degree reasons).  Hence $S$ is not a content $R$-algebra.

Suppose on the other hand that $K$ is algebraically closed in $L$.  Then all irreducible polynomials over $K$ remain irreducible over $L$.  To see this,  let $f \in K[x]$ be irreducible.  Without loss of generality, we may assume $f$ is monic. Let $L'$ be an algebraic closure of $L$, and let $K'$ be an algebraic closure of $K$ in $L'$.  Then $K'$ is algebraically closed.  Let $f=\prod_{i=1}^n (x-\alpha_i)$ be the unique factorization of $f$ into linear factors in $K'[x]$.  Say $f=gh$ with $g, h \in L[x]$ and monic.  Without loss of generality (by changing order of the factors), $g=\prod_{i=1}^s (x-\alpha_i)$ and $h=\prod_{i=s+1}^n (x-\alpha_i)$.  In particular, the coefficients of $g$ and $h$ are in $K' \cap L=K$, where the equality holds by hypothesis. That is, $g, h$ are monic polynomials in $K[x]$.  Since $f$ is irreducible in $K[x]$, either $g=f$ or $h=f$.  Thus $f$ is irreducible in $L[x]$.

Accordingly, let $\m$ be a maximal ideal of $R$. Then $\m = (f)$, where $f(x)\in K[x]$ is an irreducible polynomial.  But since $f$ is also irreducible in $L[x]$, hence in $S$, and since irreducible polynomials generate prime ideals in $S$ by virtue of $S$ being a UFD, we have $\m S = fS$ is prime.  Then by Proposition~\ref{pr:ES47}, $S$ is a content $R$-algebra.
\end{proof}

\begin{rem}
We note again the striking contrast between Proposition~\ref{pr:irredfield} (the one-variable case) and Example~\ref{ex:aclosed} (the two-variable case).
\end{rem}

\begin{thm}\label{thm:valgroups}
Let $(V,\m) \subseteq (S,\n)$ be an inclusion of valuation domains, with $\m \neq 0$.  Then $S$ is a content $V$-algebra if and only if the induced homomorphism of value groups is an isomorphism.
\end{thm}

\begin{proof}
First suppose $S$ is a content $V$-algebra.  To show that the induced homomorphism on value groups is an isomorphism, it suffices to show that the induced  monoid map from the nonnegative elements of the value group of $V$ to the nonnegative elements of the value group of $S$ is injective and surjective.

For injectivity, let $x, y \in V$ with $\nu_S(x) = \nu_S(y)$ -- i.e. $xS = yS$.  Then there is some unit $u$ of $S$ such that $x=uy$.  Hence $x \in yS \cap V = yV$ by purity, and since $u^{-1} \in S$, the equation $y=u^{-1}x$ implies that $y \in xS \cap V = xV$.  Thus, $xV = yV$; hence $\nu_V(x) = \nu_V(y)$.

For surjectivity, let $g \in S$.  Then by Proposition~\ref{pr:cval} (and its proof), we have $\orc(g) = gS \cap V = rV$ for some $r\in V$.  Thus $g \in \orc(g)S = rS$.  Say $g=rs$ for some $s\in S$.  Then $rV = \orc(g) = \orc(rs) = r\orc(s)$, so that $\orc(s) = V$.  But since $\n = \m S$, it follows that $s$ is a unit, whence $gS = rS$.  Hence, the element of the value group of $V$ corresponding to $r$ maps to the element of the value group of $S$ corresponding to $g$.

Next, we prove the converse.  So suppose the map of value groups is an isomorphism.  Flatness of the map $V \rightarrow S$ comes from the fact that it is an inclusion of domains, hence torsion-free, hence (since $V$ is Pr\"ufer) flat.  The map is \emph{faithfully} flat because for any $r\in \m$, we have $\nu_V(r)>0$, hence $\nu_S(r) >0$ so $r\in \n$.  On the other hand, for any $g \in \n$, there is some $r \in \m$ with $\nu_S(g) = \nu_V(r) = \nu_S(r)$, so that $gS = rS \subseteq \m S$.  Hence $\n=\m S$.  To complete the proof (using \cite[Proposition 4.7]{nmeSh-OR}), we need only show that the map is Ohm-Rush.  By what we have already shown, for all $g\in S$ we have $g \in (gS \cap V)S$  Thus, $\orc(g) = gS \cap V$ and $g \in \orc(g)S$.  
\end{proof}

\begin{thm}\label{thm:homeo}
Let $R \rightarrow S$ be a content algebra map, where $S$ is a valuation ring and $R$ is not a field.  Then the induced map $\Spec S \rightarrow \Spec R$ is a homeomorphism.
\end{thm}

\begin{proof}
By Lemma~\ref{lem:valdescent}, $R$ is a valuation ring.  Let $\phi: \Spec S \ra \Spec R$ be the spectral map.  We first show that $\phi$ is a bijection.  Surjectivity arises from the fact that for any $\p \in \Spec R$, $\p S \in \Spec S$ and $\phi(\p S) = \p S \cap R = \p$ by purity.

For injectivity, it suffices to show that \emph{every} prime ideal of $S$ is extended from $R$, since any two distinct primes of $R$ extend to distinct primes of $S$ by faithful flatness.  But since the maximal ideal of $S$ is extended from $R$ (by Corollary~\ref{cor:maxextend}), the claim follows from Lemma~\ref{lem:extended}.

Finally, since $\phi$ is continuous, it suffices to show that it is a \emph{closed} map.  Indeed, we claim that for any ideal $I$ of $S$, we have $\phi[$V$_S(I)] = $V$_R(I \cap R)$.  If $I \subseteq P$, then of course $I \cap R \subseteq P \cap R$.  Conversely, let $\p \in \Spec R$ with $I \cap R \subseteq \p$.  Let $g \in I$.  Then by Proposition~\ref{pr:cval}, \[
g \in \orc(g)S = (gS \cap R)S \subseteq (I \cap R)S \subseteq \p S.
\]
Hence $I \subseteq \p S$, so $\p S \in $V$_S(I)$ and $\phi(\p S) = \p S \cap R = \p$.  
\end{proof}

\begin{rem}
We note that in the above, we need not have $V=S$.    As a simple example let $K\subset L$ be distinct fields. Then $S=L[[x]]$ is a content algebra over $K[[x]]$.

Similarly, let $G$ be any ordered group, let $K \subseteq L$ be a field extension, consider the field extension $K(G) \rightarrow L(G)$, and let $R$, $S$ be the subrings of $K(G)$, $L(G)$ corresponding to the elements of nonnegative $G$-value.  Then by Theorem~\ref{thm:valgroups}, the inclusion $R \subseteq S$ of valuation domains is a content algebra map, and hence the map $\Spec S \rightarrow \Spec R$ is a homeomorphism.

On the other hand, the result is special to valuation domains.  For consider the map $R=K[x] \rightarrow L[x]=S$ from Example~\ref{ex:polynomial} (i.e., with $L$ purely transcendental over $K$).  As long as $K \neq L$, this map does not induce a homeomorphism or even a bijection of prime spectra.  For choose some monic irreducible polynomial $g \in L[x] \setminus K[x]$ (e.g. one can always choose $x-u$ where $u\in L \setminus K$).  Then $P = gS$ is prime, and $P \cap R = 0R = 0S \cap R$, so the Spec map is not injective.  It is also not a \emph{closed} map, since $0R \in \phi(V_S(P))$, but $\Spec R = V_R(0) \nsubseteq \phi(V_S(P))$.

One can even create an example where the \emph{base} ring is a valuation ring by setting $W := R \setminus \m$, where $\m$ is a maximal ideal of $R$.  Then in the $\Spec$ map, the maximal ideal $\m S_W$ contracts to $\m R_W$, but all other maximal ideals of $S_W$ contract to $0R_W$, so again the Spec map is not injective and fails to be closed.
\end{rem}

\begin{example}\label{ex:end}
 We present a method to construct semilocal B\'ezout domains that are themselves content algebras over semilocal B\'ezout domains. These examples do not fall into any of the classes explored thus far.

The general construction is as follows.  Let $L/K$ be a field extension.  Let $V_1, \ldots, V_n$ (resp. $W_1, \ldots, W_n$) be pairwise independent valuation rings with fraction field $K$ (resp. $L$), in such a way that each $V_i \subseteq W_i$ and $W_i$ is a content $V_i$-algebra. (Recall 
that a pair $V, V'$ of valuation rings with common fraction field $F$ are \emph{independent} if there is no ring properly contained in $F$ that contains both $V$ and $V'$.) Set $R := \bigcap_{i=1}^n V_i$ and $S := \bigcap_{i=1}^n W_i$.  We then claim that $R$, $S$ are Pr\"ufer domains (indeed Dedekind domains if each of the $V_i$ is a DVR), and $S$ is a content $R$-algebra.

It follows from \cite[Theorem 107]{Kap-CR} that $R$ and $S$ are B\'ezout domains (indeed principal ideal domains if all the $V_i$ and $W_i$ are DVRs, but by Theorem~\ref{thm:valgroups}, the $W_i$ are DVRs whenever the $V_i$ are).   If $\m_i$ (resp. $\n_i$) is the maximal ideal of $V_i$ (resp. $W_i$) for $1\leq i \leq n$, then the maximal ideals of $R$ (resp. $S$) are $M_i = \m_i \cap R$ (resp. $N_i = \n_i \cap S$), and we have $R_{M_i} = V_i$ (resp. $S_{N_i} = W_i$).  It is then easily seen from independence that any nonzero prime ideal is contained in a unique maximal ideal.   Furthermore, since $S$ is torsion free over $R$, it is flat over $R$. Additionally, since $M_iS \subseteq N_i$, it follows that $S$ is faithfully flat over $R$.

We next show that $S$ is Ohm-Rush over $R$.   Let $I$ be an ideal of $R$ with $g\in IS$, whence $gS \cap R \subseteq IS \cap R = I$ by purity.  Therefore, $ gS\cap R \subseteq  \orc(g)$.  To finish we will show that $g\in (gS\cap R)S$, from which it follows that $ gS\cap R =  \orc(g)$ and that $S$ is Ohm-Rush over $R$.

To this end let  $X_i := R \setminus M_i$, and note that $S_{X_i} = S_{N_i}$.  To see this, it is enough to show that $S_{X_i}$ has unique maximal ideal $N_iS_{X_i}$.  Accordingly, let $P$ be a prime ideal of $S$ such that $PS_{X_i}$ is maximal.  Let $N$ be the unique maximal ideal of $S$ that contains $P$.  Then $P \cap R \subseteq N \cap R$, but also $P \cap R \subseteq M_i$, so that since only one maximal ideal of $R$ can contain $P\cap R$, $M_i = N \cap R$.  Hence $N = N_i$.

Then by Proposition~\ref{pr:cval}, since each $W_i$ is a valuation ring that is content over $V_i$, we have $g \in (gW_i \cap V_i)W_i = (g S_{N_i} \cap R_{M_i})S_{N_i} = (gS_{X_i} \cap R_{X_i})S_{X_i} = ((gS \cap R)S)_{X_i}$, with the last equality following by flatness of localization.  Thus, thought of as $R$-modules, we have $(gS)_{M_i} \subseteq ((gS \cap R)S)_{M_i}$ for all $i$, but also it is clear that $(gS \cap R)S \subseteq gS$.  Since they are locally equal at all maximal ideals of $R$, they are equal.  That is, $gS = (gS \cap R)S$, so $g\in (gS\cap R)S$.

Finally, by Proposition~\ref{pr:ES47}, to show that $S$ is a content $R$-algebra it will suffice to show that maximal ideals of $R$ extend to primes of $S$.  Let $M$ be a maximal ideal of $R$. Then $M = M_i = \m_i \cap R$ for some $1\leq i \leq n$.  By Corollary~\ref{cor:maxextend}, we have $\m_i W_i = \n_i$.  Thus, as before, setting $X_i := R \setminus M_i$, we have \[
(MS)_{X_i} = M_{X_i} S_{X_i} = \m_i W_i = \n_i = N_i S_{N_i} = N_i S_{X_i}.
\]
But also for $j\neq i$, we have $(MS)_{X_j} = M_{X_j} S_{X_j} = W_j = S_{N_j} = N_i S_{N_j} = N_i S_{X_j}$.  On the other hand, $MS = (\m_i \cap R)S \subseteq \n_i \cap S = N_i$.  Therefore again we have locally equal $R$-modules with a containment, so $MS = N_i$, a prime (indeed maximal) ideal of $S$.

When does such independence happen?  For one thing, the valuation rings given by the localizations of a 1-dimensional Pr\"ufer domain (e.g. a Dedekind domain) are always pairwise independent, as there are no local rings between a rank 1 valuation ring and its fraction field.  Moreover, since the resulting rings will be semilocal, they will be B\'ezout (resp. principal ideal) domains.

We next present a concrete example from which one can also see how to construct examples of arbitrary finite rank.

Let $\ell/k$ be a field extension, let $r,s,t,u$ be indeterminates over $\ell$, set $K:= k(r,s,t,u)$ and $L := \ell(r,s,t,u)$.  Let $v_1$ (resp. $w_1$) be the lexicographic $\mathbb Z^2$-valued valuation on $K$ (resp. $L$) determined by $v_1(t) =v_1(u) = {\bf 0} = w_1(t) = w_1(u)$, while $v_1(r) = (1,0) = w_1(r) $, and $v_1(s) =(0,1)= w_1(s)$.  

 Similarly, let $v_2$ (resp. $w_2$) be the lexicographic $\mathbb Z^2$-valued valuation on $K$ (resp. $L$) that is ${\bf 0}$ on $k(r,s)^\times$ (resp. $\ell(r,s)^\times$), while $v_2(t)=(1,0)=w_2(t)$ and  $v_2(u) =(0,1)=w_2(u)$.  Let $V_1, V_2, W_1, W_2$ be the valuation rings associated to the valuations $v_1, v_2, w_1, w_2$ respectively.  Then $V_1$ and $V_2$ are pairwise independent, because if $U$ is a local subring of $K$ that contains both $V_1$ and $V_2$, then since $U$ contains $V_2$, both $r$ and $s$ are units of $U$, but since $U$ also contains $V_1$, it follows that $U=K$.  Similarly, $W_1$ and $W_2$ are independent.  Moreover, the natural maps $V_i \rightarrow W_i$, $i=1,2$, are content algebra maps by Theorem~\ref{thm:valgroups}.  Then by the above, $W_1 \cap W_2$ is a content $(V_1 \cap V_2)$-algebra.
\end{example}

\section*{Acknowledgments}
We are grateful to the anonymous referee for a thorough reading and comments that improved the paper.  In particular, the material in Section~\ref{sec:localglobal} was nowhere near as extensive before the referee made his/her insightful comments. Several of the results therein are essentially due to the referee.

\providecommand{\bysame}{\leavevmode\hbox to3em{\hrulefill}\thinspace}
\providecommand{\MR}{\relax\ifhmode\unskip\space\fi MR }
\providecommand{\MRhref}[2]{%
  \href{http://www.ams.org/mathscinet-getitem?mr=#1}{#2}
}
\providecommand{\href}[2]{#2}

\end{document}